\documentclass[reqno,11pt, a4] {amsart}
\usepackage[active]{srcltx}
\usepackage[poly,curve,line, frame, all, arc]{xy}
\usepackage{pb-diagram}
\usepackage{epstopdf}
\usepackage{pdfsync}

\usepackage{enumitem}
\usepackage{mathrsfs}
\usepackage{amsmath}
\usepackage{amssymb}
\usepackage{amscd}
\usepackage{amsthm}
\usepackage[latin1]{inputenc}
\usepackage{graphics}
\usepackage{graphicx}
\usepackage{varioref}

\labelformat{enumi}{(#1)}

\newtheorem{teo}{Theorem}[section]
\newtheorem{defin}{Definition}[section]
\newtheorem{remark}{Remark}[section]
\newtheorem{prop}{Proposition}[section]
\newtheorem{cor}{Corollary}[section]
\newtheorem{lemma}{Lemma}[section]

\newtheoremstyle{dico}
 {\baselineskip}   
  {\topsep}   
  {}  
  {0pt}       
  {} 
  {.}         
  {5pt plus 1pt minus 1pt} 
  {}          
\theoremstyle{dico}
\newtheorem{say}[equation]{}
\numberwithin{equation}{section}

\newcommand{\lds}{\ldots}
\newcommand{\cds}{\cdots}
\newcommand{\cd}{\cdot}

\newcommand{\om}{\omega}

\renewcommand{\phi}{\varphi}

\newcommand{\ra}{\rightarrow}
\newcommand{\lra}{\longrightarrow}
\newcommand{\C}{\mathbb{C}}
\newcommand{\R}{\mathbb{R}}

\newcommand{\su} {\mathfrak{su}}

\newcommand{\restr}[1]          {\vert_{#1}}

\newcommand{\meno}{^{-1}}

\newcommand{\PP}{\mathbb{P}}

\newcommand{\conv} {\operatorname{conv}}
\newcommand{\liu}{\mathfrak{u}}

\newcommand{\lia}{\mathfrak{a}}

\newcommand{\lieg}{\mathfrak{g}}

\newcommand{\lieb}{\mathfrak{b}}
\newcommand{\liep}{\mathfrak{p}}

\newcommand{\liet}{\mathfrak{t}}

\newcommand{\vacuo}{\emptyset}
\newcommand{\sx}{\langle}                   
\newcommand{\xs}{\rangle}

\newcommand{\scalo}{\sx \, , \, \xs}
\newcommand{\noparty}[1]{}
\newcommand{\mup}{\mu_\liep}
\newcommand{\mua}{\mu_\lia}
\newcommand{\mupb}{\mu_\liep^\beta}

\newcommand{\fun}{\mathfrak{F}}

\newcommand{\spaz}{\mathscr{M}}

\newcommand{\desudt}[1][]{\dfrac {\mathrm {d} #1 }{\mathrm {dt}}}
\newcommand{\desudtzero}{\desudt \bigg \vert _{t=0} }

\newcommand{\proba}{\mathscr{P}}
\newcommand{\pb}{\proba(X)}

\newcommand{\PsiM}{\Psi^\proba}

\newcommand{\PC}{\PP^n(\C)}
\newcommand{\wm}{W_{\max}}

\begin{document}
\title{Remarks on the abelian convexity theorem}
\author{Leonardo Biliotti}
\author{Alessandro Ghigi}

\begin{abstract}
  This note contains some observations on abelian convexity theorems.
  Convexity along an orbit is established in a very general setting
  using Kempf-Ness functions.  This is applied to give short proofs of
  the Atiyah-Guillemin-Sternberg theorem and of abelian convexity for
  the gradient map in the case of a real analytic submanifold of
  complex projective space. Finally we give an application to the
  action on the probability measures.
\end{abstract}

\address{Universit\`{a} degli Studi di Parma} \email{leonardo.biliotti@unipr.it}
\address{Universit\`a degli Studi di Pavia}\email{alessandro.ghigi@unipv.it}

\thanks{The authors were partially supported by FIRB 2012 ``Geometria
  differenziale e teoria geometrica delle funzioni'' and by GNSAGA of
  INdAM.  The first author was also supported by MIUR PRIN 2015 ``Real
  and Complex Manifolds: Geometry, Topology and Harmonic Analysis''.
  The second author was also supported by MIUR PRIN 2015 ``Moduli
  spaces and Lie theory''. }



\maketitle

\section{Introduction}
\label{sec:introduction}

\begin{say}
  \label{caso-complesso}
  Let $U$ be a compact connected Lie group and let $U^\C$ be its
  complexification.  Let $(Z,\omega)$ be a K\"ahler manifold on which
  $U^\C$ acts holomorphically. Assume that $U$ acts in a Hamiltonian
  fashion with momentum map $\mu:Z \lra\liu^*$.  This means that $\om$
  is $U$-invariant, $\mu$ is equivariant and for any $\beta \in \liu$
  we have
  \begin{gather*}
    d\mu^\beta = i_{\beta_Z} \om,
  \end{gather*}
  where $\mu^\beta = \sx \mu, \beta \xs$ and $\beta_Z$ denotes the
  fundamental vector field on $Z$ induced by the action of $U$.  It is
  well-known that the momentum map represents a fundamental tool in
  the study of the action of $U^\C$ on $Z$. Of particular importance
  are convexity theorems \cite{atiyah-commuting,
    guillemin-sternberg-convexity-1,kirwan}, which depend on the fact
  that the functions $\mu^\beta$ are Morse-Bott with even indices.
\end{say}

\begin{say}\label{gradient-cl}
  More recently the momentum map has been generalized to the following
  setting
  \cite{heinzner-schwarz,heinzner-schwarz-stoetzel,heinzner-stoetzel,heinzner-schuetzdeller}.
  Let $G\subset U^\C$ be a closed connected subgroup of $U^\C$ that is
  \emph{compatible} with respect to the Cartan decomposition of
  $U^\C$.  This means that $G$ is a closed subgroup of $U^\C$ such
  that $G=K\exp (\liep)$, where $K=U\cap G$ and
  $\liep=\lieg \cap i\liu$
  \cite{heinzner-schwarz-stoetzel,heinzner-stoetzel-global}.  The
  inclusion $i \liep \hookrightarrow \liu$ induces by restriction a
  $K$-equivariant map $\mu_{i \liep}:Z \lra (i \liep)^*$.  Using a
  fixed $U$-invariant scalar product $\scalo$ on $\liu$, we identify
  $\liu \cong \liu^*$.  We also denote by $\scalo$ the scalar product
  on $i\liu$ such that multiplication by $i$ is an isometry of $\liu$
  onto $i\liu$.  For $z \in Z$ let $\mup (z) \in \liep$ denote $-i$
  times the component of $\mu(z)$ in the direction of $i\liep$.  In
  other words we require that
  \begin{gather}
    \label{mup}
    \sx \mup (z) , \beta \xs = - \sx \mu(z) , i\beta\xs,
  \end{gather}
  for any $\beta \in \liep$. The map $ \mu_\liep : Z \ra \liep $ is
  called the $G$-\emph{gradient map}.  Given a compact $G$-stable
  subset $X \subset Z$ we consider the restriction
  $\mup:X \lra \liep$.  We also set
  \begin{gather*}
    \mupb:= \sx \mup, \beta \xs = \mu^{-i\beta}.
  \end{gather*}
  Many fundamental theorems regarding the momentum map hold also for
  the gradient map.  The functions $\mup^\beta$ are Morse-Bott,
  although in general not with even indices.  Even so in
  \cite{heinzner-schuetzdeller} (see also \cite{bghc}) the authors
  prove the following convexity theorem: let $V$ be a unitary
  representation of $U$ and let $Y \subset \mathbb {P}(V)$ be a closed
  real semi-algebraic subset, whose real algebraic Zariski closure is
  irreducible.  If $\mathfrak a$ is a maximal abelian subalgebra of
  $\lieg$ contained in $\liep$ and $\mathfrak a_{+}$ is a positive
  Weyl chamber, then $\mup(Y) \cap \mathfrak a_{+}$ is a convex
  polytope. The proof is rather delicate.

  One of the goals of the present note is to give a convexity theorem
  along an orbit, i.e. to show that the image of an orbit via the
  gradient map is convex. This will be proved in a very general
  setting using only so-called Kempf-Ness functions.  This allows us
  to prove the corresponding theorem for the gradient map without
  using results from the complex case. As applications we get a simple
  proof of the abelian convexity theorem for the gradient map for real
  analytic submanifolds and the convexity along an orbit for the
  gradient map associated to the induced action on probability
  measures.
\end{say}

\begin{say}\label{abst}
Using the same notation as above, assume that $X$ is a
    compact $G$-invariant submanifold of $Z$.  In \cite{bgs,bz} the
    authors and Zedda studied the action of $G$ on the set of
    probability measures on $X$.  This set is not a manifold, but
  many features of the action, especially those relating only to a
  single orbit closure, can be studied with a formalism very similar
  to the momentum map. We now recall this formalism.

  Let $\spaz$ be a Hausdorff topological space and let $G$ be a
  non-compact real reductive group which acts continuously on $\spaz$.
  We can write $G=K\exp (\liep)$, where $K$ is a maximal compact
  subgroup of $G$.  Given a function $ \Psi : \spaz \times G \ra \R$,
  consider the following properties.
  \begin{enumerate}[label=(P\arabic*),ref=P\arabic*]
  \item \label{P1} For any $x\in \spaz$ the function $ \Psi(x,\cd )$
    is smooth on $G$.
  \item \label{P2}The function $\Psi(x, \cd )$ is left--invariant with
    respect to $K$, i.e.: $\Psi(x,kg) = \Psi(x,g)$.
  \item \label{P3}For any $x\in \spaz$, and any $\xi \in \liep$ and
    $t\in \R${:}
    \begin{gather*}
      \frac{\mathrm{d^2}}{\mathrm{dt}^2 } \Psi(x,\exp(t\xi)) \geq 0.
    \end{gather*}
    Moreover:
    \begin{gather*}
      \frac{\mathrm{d^2}}{\mathrm{dt}^2 }\bigg \vert_{t=0}
      \Psi(x,\exp(t\xi)) = 0
    \end{gather*}
    if and only if $\exp(\R \xi) \subset G_x$.
  \item For any $x\in \spaz$, and any $g, h\in G$:
    \begin{gather*}
      \Psi(x,g) + \Psi({gx}, h) = \Psi(x,hg).
    \end{gather*}
    This equation is called the \emph{cocycle condition}.
  \end{enumerate}
  In order to state our fifth condition, let
  $\scalo : \liep^*\times \liep \ra \R$ be the duality pairing.  For
  $x\in \spaz$ define $\fun(x) \in \liep^*$ by requiring that:
  \begin{equation}
    \label{momento-astratto}
    \sx \fun (x), \xi \xs =   \desudtzero \Psi(x,
    \exp(t\xi)) .
  \end{equation}
  \begin{enumerate}[label=(P\arabic*),ref=P\arabic*]
    \setcounter{enumi}{4}
  \item \label{P5} The map $\fun : \spaz \ra \liep^*$ is continuous.
  \end{enumerate}
\end{say}
\begin{defin}
  \label{def-kn}
  Let $G$ be a non-compact real reductive Lie group, $K$ a maximal
  compact subgroup of $G$ and $\spaz$ a Hausdorff topological space
  with a continuous $G$--action. A \emph{Kempf-Ness function} for
  $(\spaz, G,K)$ is a function
  \begin{gather*}
    \Psi : \spaz \times G \ra \R ,
  \end{gather*}
  that satisfies conditions \ref{P1}--\ref{P5}.  The map $\fun$ is
  called the \emph{gradient map} of $(\spaz, G, K, \Psi).$
\end{defin}
By \cite[Prop. 5]{bz} $\fun : \spaz \ra \liep^*$ is a $K$-equivariant
map.  Since $K$ is compact, we may fix a $\mathrm{K}$-invariant scalar
product $\langle \cdot,\cdot \rangle$ of $\liep$ and we may identify
$\liep^* \cong \liep$ by means of $\langle \cdot,\cdot \rangle$.
Hence we may think the gradient map as a $\liep$-valued map
$\fun:\spaz \ra \liep$.

\begin{remark}
  In \cite{bgs,bz} a sixth hypothesis is assumed, which is necessary
  to define the maximal weight and to deal with stability issues. This
  hypothesis is not needed for the arguments of the present paper.
\end{remark}

\begin{say}
The original setting \cite{kempf-ness} for what we call
    Kempf-Ness function is the following: let $V$ be a unitary
    representation of $U$.  For $x=[v] \in \PP(V)$ and $g\in U^\C$ set
    $\Psi(x,g):= \log ( |g\meno v | / |v|)$. This function satisfies
    \ref{P1}--\ref{P5} with $\fun=\mu$, the momentum map.
    Thus the behaviour of the momentum map is encoded in the function
    $\Psi$.  Functions similar to these ones exist for rather general
    actions.  The following result has been proven in \cite[\S
    2]{hhinv}, \cite{azad-loeb-bulletin}, \cite{mundet-Crelles} for
    $G=U^\C$ and in \cite[\S 7]{bz} in the general case.
\end{say}
\begin{prop}\label{Kempf-Ness-gradient-peter}
  Let $X,G,K$ be as in \ref{gradient-cl}. Then there exists a
  Kempf-Ness function $\Psi$ for $(X, G, K)$ satisfying the conditions
  $(P1)-(P5)$ such that $\fun=\mup$.
\end{prop}

\begin{say}
  In the present note we study abelian convexity theorems. In \S
  \ref{sec:abelian-convexity} we give an easy proof of convexity for
  the image of an orbit of an abelian group in the setting of
  Kempf-Ness functions, see Theorem \ref{conv2}. In \S 3 we apply this
  to the setting of the gradient map as in \ref{gradient-cl}.

  If $G=A=\exp (\lia)$, where $\lia \subset \liep$ is an abelian
  subalgebra, we are able to prove that the image of the gradient map
  of an $A$-orbit, is convex (Theorem \ref{conv-orbite}) without using
  the convexity results available in the complex setting (see
  \cite[p. 5]{heinzner-stoetzel}). Our proof only uses the existence
  of Kempf-Ness functions.

  We also give a new proof of the Atiyah-Guillemin-Sternberg convexity
  theorem.  Indeed consider the case where $X=Z$ is compact, $T$ is a
  compact torus and $G=T^\C$.  Atiyah \cite{atiyah-commuting}
  suggested that the convexity of $\mu (T^\C\cd p)$ (for $p\in Z)$
  could be used to give an alternative proof of the abelian convexity
  theorem showing that there always exists $p\in Z$ such that
  $ \overline{\mu(T^\C\cd p) } = \mu(Z)$.  Duistermaat
  \cite{duistermaat} proved that the set of points $p$ with
  $ \overline{\mu(T^\C\cd p) } = \mu(Z)$ is non-empty and dense (see
  also \cite{bloch-ratiu}) We give a new proof of this result and we
  also show that this set is open.  More importantly, we believe that
  the abstract approach that we follow adds to the understanding of
  some basic results in the subject.

In the case of real analytic submanifolds of $\PP^n(\C)$ our method
  yields the following.
\end{say}
\begin{teo}
    Let $X \subset \PC$ be a compact connected real analytic
    submanifold that is invariant by $A = \exp (\lia)$ where
    $\lia \subset i \su(n+1)$ is an abelian subalgebra.
    Then
    \begin{enumerate}
    \item $\mua(X)$ is a convex polytope with vertices in $\mu(X^A)$;
    \item the set $\{x\in X:\, \mua( \overline{A\cdot x})=\mua (X)\}$
      is open and dense;
    \item for any face $\sigma \subset\mua(X)$, there is an $A$-orbit
      $Y$ such that $\mua (\overline{Y} )=\sigma$.
    \end{enumerate}
  \end{teo}
  This result is weaker than the one obtained by Heinzner and
  Sch\"utzdeller \cite{heinzner-schuetzdeller} (even in the abelian
  case). Nevertheless the proof in this note is very simple, and $(b)$
  and $(c)$ are new.  So we think that this might be of some interest.

  In the last section we apply the result of \S
  \ref{sec:abelian-convexity} to the action of $G$ on the set of
  probability measures on $X$ (with the notations of \ref{abst}).
  This yields a simpler and more natural proof of the convexity
  theorem for measures obtained in \cite{bilio-raffero}.

\medskip

{\bfseries \noindent{Acknowledgements}}.  The authors wish to thank
Peter Heinzner for many important discussions and explanations related
to the subject of this paper. They also would like to thank Mich\`ele
Vergne for pointing out reference \cite{kacpe}. Finally they are
grateful to the anonymous referee for a very carefully  reading of the
manuscript.

\section{Abstract Abelian convexity}
\label{sec:abelian-convexity}

The following Proposition contains the key idea and it is basic to the
whole paper.  Let $\spaz, G, K, \liep, \Psi$ and $\fun$ be as in
\ref{abst}.  Let $\lia \subset \liep$ be an abelian subalgebra. Then
$A := \exp (\lia) \subset G$ is a compatible abelian subgroup.

\begin{prop}\label{convex}
  Let $\Psi:\spaz \times A \ra \R$ be a Kempf-Ness function for
  $(\spaz, A,\{e\})$ and let $\fun:\spaz \ra \lia$ be the
  corresponding gradient map. Let $x\in \spaz$ and let
  $A_x=\exp(\lia_x)$ be the stabilizer of $x$.  Let
  $\pi:\lia \ra \lia_x^\perp$ be the orthogonal projection. Then
  $\pi (\fun(A\cd x))$ is an open convex subset of
  $\lia_x^\perp$. Moreover, $\fun(A\cd x) $ is an open convex subset
  of $\fun(x) + \lia_x^\perp$.
\end{prop}
\begin{proof}
  Set $\lieb:=\lia_x^\perp$ and consider the function
  $f:\lieb \ra \R, f(v)=\Psi(x,\exp(v))$.  Fix $v ,w\in \lieb$ with
  $w \neq 0$ and consider the curve $\gamma(t)=v+tw$.  Set
  $u(t)=f(\gamma(t))$. We claim that $u''(0) >0$. Using the fact that
  $A$ is abelian, the cocycle condition yields
  \begin{equation*}
    \begin{split}
      u(t)&=\Psi(x,\exp(v+tw))=\Psi(x,\exp(tw)\exp(v))\\
      &=\Psi(\exp(v)x,\exp(tw))+\Psi(x,\exp(v)),
    \end{split}
  \end{equation*}
  so
$$
u'(t)=\frac{\mathrm d}{\mathrm{dt}} \Psi(\exp(v)x,\exp(tw)).
$$
By \ref{P3} we have $u''(0)\geq 0$ and the equality would imply that
{$w \in \lia_{\exp(v)x} = \lia_x$}, which is impossible since
$w \in \lia_x^\perp$.  This proves the claim and shows that $f$ is a
strictly convex function on $\lieb$.  Therefore, by basic result in
convex analysis \cite[p.122]{gt}, $ \mathrm d f (\lieb )$ is an open
convex subset of $\lieb\cong (\lieb)^*$ Moreover the computation above
also shows that
\begin{equation}
  \label{zz1}
  (\mathrm d f)_v (w)= \langle \fun (\exp(v)x),w \rangle=\langle \pi(\fun (\exp(v)x)),w \rangle.
\end{equation}
Using the fact that $A x = \exp(\lieb) x$ we conclude that
\begin{equation*}
  \pi(\fun (A x))=\pi(\fun (\exp(\lieb) x))=\mathrm d f (\lieb )
\end{equation*}
is an open convex set of $\lieb$.  This proves the first assertion.
To prove the last assertion it is enough to check that for any
$v\in \lia $ and for any $w\in \lia_x$
\begin{gather}
  \label{eq:zz1}
  \sx \fun(\exp(v)\cd x), w \xs = \sx \fun (x), w\xs.
\end{gather}
Using \eqref{momento-astratto} and the cocycle condition we have
\begin{gather*}
  \sx \fun(\exp(v)\cd x), w \xs =
  \desudtzero \Psi (\exp(v)\cd x, \exp(tw)) = \\
  =\desudtzero \Psi ( x, \exp(tw)\exp(v) ).
\end{gather*}
Using that $v$ and $w$ commute and again the cocycle condition and
\eqref{momento-astratto} we get
\begin{gather*}
  \sx \fun(\exp(v)\cd x), w \xs =
  \desudtzero \Psi ( x, \exp(v)\exp(tw) )  = \\
  = \desudtzero \biggl ( \Psi ( \exp(tw)\cd x, \exp(v) ) + \Psi( x,
  \exp(tw)) \biggr )
  =  \\
  = \desudtzero \Psi ( x, \exp(v) )+ \sx \fun(x), w\xs = \sx \fun(x),
  w\xs.
\end{gather*}
This proves \eqref{eq:zz1}.
\end{proof}
\begin{cor}
Let $x\in \spaz$ be such that $A_x=\{e\}$. Then $\fun(A\cd x)$ is an
  open convex set of $\lia$.
\end{cor}
\begin{cor}
   Set $E:=\pi (\fun(\overline{A\cd x})) $.
If
$y \in
 \overline{A\cd
    x}$,
and  $p:=\pi(\fun(y)) \in \partial E$,
 then $\lia_x\subsetneq \lia_y$.
\end{cor}
\begin{proof}
  Since the $A$-action
  on $\spaz$
  is continuous, it follows $A_x\subset
  A_y$ and so $\lia_y^\perp \subset
  \lia_x^\perp$. Assume by contradiction that
  $\lia_x=\lia_y$
  and let $\pi:\lia
  \lra \lia_x^\perp$ be the orthogonal projection on
  $\lia_x^\perp$.
  By Proposition \ref{convex} the set $\Omega:=\pi(\fun
    (A\cd y))$ is an open convex subset of $\lia_x^\perp$. Since $A\cd
    y \subset \overline{A\cd x}$, we have $p\in \Omega \subset
    E$.   But this contradicts the fact that $p\in \partial
  E$. Thus $\lia_x\subsetneq \lia_y$.
\end{proof}
\begin{teo}\label{conv2}
  If $\overline{A\cd x} $ is compact, then
  \begin{gather*}
    \overline{\fun(A\cd x)} = \fun(\overline{A\cd
      x})=\mathrm{conv}\big(\fun(\overline{A\cd x}\cap \spaz^A)\big).
  \end{gather*}
\end{teo}
\begin{proof}
  Since $\overline{A\cd
    x }$ is compact $\fun (\overline{A\cd x} ) = \overline { \fun
    (A\cd x)}$.  By Proposition \ref{convex} $E:=\fun(A\cd
  x)$ is an open convex subset of the affine subspace
  $L:=\fun(x)+\lia_x^\perp$,
  while $\bar{E}=
  \fun (\overline{A\cd x})$ is a compact convex subset.  Let $p\in
  \bar{E}$ be an extreme point and let $y\in \overline{A\cd
    x}$ be such that
  $\fun(y)=x$.
  Again by Proposition \ref{convex} $\fun(A\cd
  y)$ is a convex subset of dimension equal to $\dim
  \lia_y^\perp$.  Since
  $p$
  is an extreme point, this dimension must be $0$,
  so $\lia_y^\perp=\{0\}$
  and $y$
  is a fixed point of $A$.
  So the extremal points of $E$
  are contained in $\fun(
  \overline{A\cd x} \cap \spaz^A)$.  This proves the theorem.
\end{proof}

\section{Application to the gradient map}

In this section we assume that $Z,X,G,K$ be as in \ref{gradient-cl}.
Moreover we assume that $A = \exp (\lia)$, where $\lia \subset \liep$
is an abelian subalgebra.

Applying Theorem \ref{conv2} we get a new proof of the following
result.

\begin{teo}\label{conv-orbite}
  Assume that $X\subset Z$ is an $A$-invariant compact submanifold.
  For any $x\in X$, $\mu_{\lia}(A\cd x)$ is an open convex subset of
  $\mu_{\lia}(x)+\lia_x^\perp$, its closure coincides with
  $\mua(\overline{A\cd x})$, it is a polytope and it is the convex
  hull of $\mup(X^A\cap \overline{A\cd x})$.
\end{teo}
\begin{proof}
  By Proposition \ref{Kempf-Ness-gradient-peter}, there exists a
  Kempf-Ness function $\Psi$ for $(X, G, K)$ satisfying the conditions
  $(P1)-(P5)$ and such that $\fun=\mup$. Now, that
  $\mu_{\lia}(A\cd x)$ is an open convex subset of
  $\mu_{\lia}(x)+\lia_x^\perp$ is proven in Proposition \ref{convex}.
  That
  $\mua (\overline{A\cd x}) = \overline{\mua(A\cd x)} = \conv (\mua
  (\overline{A\cd x}\cap X^A)$
  is proven in Theorem \ref{conv2} (recall that $X$ is compact by
  assumption).  Next observe that $X^A$ has finitely many connected
  components, since $X$ is a compact manifold, and $\mua$ is constant
  on each of them.
\end{proof}

This convexity theorem along the orbits was proven by Atiyah
\cite{atiyah-commuting} in the case, where $X=Z$ and $A$ is a complex
torus.  The general case has been proven by Heinzner and St\"otzel
\cite[Prop. 3]{heinzner-stoetzel}.  The above proof via Theorem
\ref{conv2} is quite short.  Note that the first statement in Theorem
\ref{conv-orbite}, i.e. that $\mu_{\lia}(A\cd x)$ is an open convex
subset of $\mu_{\lia}(x)+\lia_x^\perp$, works even if $X$ is not
compact. We mention that a simple proof of orbit convexity for an
action of a complex torus on a projective manifold can be found in
\cite{kacpe}, see also \cite[p. 44]{fudan}.

Next we turn to the abelian convexity theorem.  Fix an abelian
subalgebra $\lia \subset \liep$ and set $A:= \exp (\lia)$.  Given a
subset $X\subset Z$ and $\beta \in \lia$ set
\begin{gather}
  \label{defW}
  W_{\max}^\beta(X) : = \{ x \in X: \lim_{t \to +\infty } \mua^\beta
  (\exp(t\beta) \cd x ) = \max_X \mua^\beta \}.
\end{gather}

\begin{prop}\label{caballo}
  Assume that $Z$ is compact and let $X \subset Z$ be a closed
  $A$-invariant subset.  Assume that for any $\beta \in \lia $ the set
  $W^\beta_{\max}(X)$ is open and dense in $X$.  Then
  \begin{enumerate}
  \item $P=\mua(X)$ is a convex polytope with vertices in $\mua(X^A)$;
  \item the set $\{x\in X:\, \mua( \overline{A\cdot x})=\mua (X)\}$ is
    dense and it is also open if $X$ is a smooth submanifold of $Z$;
  \item if $\sigma \subset \mua (X)$ is a face of $P$ there exists a
    $A$-orbit $Y$ such that $\mua (\overline{Y} )=\sigma$.
  \end{enumerate}
\end{prop}
\begin{proof}
  The set $Z^A$ has finitely many connected components since $Z$ is
  compact, and each component is a smooth submanifold of $Z$. Moreover
  $\mua$ is constant on each component. Therefore $\mua(Z^A)$ is a
  finite set. Since $X^A = X \cap Z^A$, we conclude that also
  $\mua(X^A)$ is a finite set.  Therefore $P:= \conv (\mua(X^A))$ is a
  convex polytope.  By Theorem \ref{conv2} if $x \in X$, then
  $\mua(\overline{A\cd x}) = \conv (\overline{A\cd x} \cap X^A)
  \subset P$.
  Hence $\conv (\mua(X) ) \subset P$. The reverse inclusion is
  obvious, so $P=\conv \mua(X)$.  Now let $\xi_1, \lds, \xi_k$ be the
  vertices of $P$. Choose $\beta_i \in \lia$ such that
  \begin{gather*}
    \{\xi \in P: \sx \xi , \beta_i\xs = \max_P \sx \cd , \beta_i\xs \}
    = \{\xi_i\}.
  \end{gather*}
  By our assumption the set
  $ W^{\beta_1}_{\mathrm{max}}\cap \cdots \cap
  W^{\beta_k}_{\mathrm{max}}$
  is open and dense. Fix
  $x\in W^{\xi_1}_{\mathrm{max}}\cap \cdots \cap
  W^{\xi_k}_{\mathrm{max}}$ and set
  \begin{gather*}
    y_i := \lim_{t \to +\infty} \exp(t\beta_i ) \cd x.
  \end{gather*}
  Then $y_i \in X$ and using \eqref{mup} and \eqref{defW} we get
  \begin{gather*}
    \mua^{\beta_i}(y_i) = \max _X \mua^{\beta_i} = \max_{\mua(X) } \sx
    \cd , \beta_i\xs = \max_{P } \sx \cd , \beta_i\xs.
  \end{gather*}
  Therefore $\mua(y_i) = \xi_i$.  So
  $\xi_i \in \mua(\overline{A\cd x} )$ for any $i=1, \lds, k$. But
  $\mua(\overline{A\cd x})$ is convex by Theorem \ref{conv2}.  Since
  $\mua(\overline{A\cd x}) \subset \mua(X) \subset P$, we get
  \begin{gather*}
    \mua(\overline{A\cd x} ) = \mua (X) = P.
  \end{gather*}
  This proves (a).  Next set
  \begin{gather*}
    W:= \{x\in X:\, \mua( \overline{A\cdot x})=\mua (X)\}.
  \end{gather*}
  We have just proven that $W$ contains
  $ W^{\xi_1}_{\mathrm{max}}\cap \cdots \cap
  W^{\xi_k}_{\mathrm{max}}$,
  so it is dense.  Assume now that $X$ is a smooth submanifold of $Z$.
  Fix one of the vertices of $P$, say $\xi_i$ and consider the set
  \begin{gather*}
    \Omega_i := \{x\in X: \overline{A\cdot x} \cap \mua^{-1}(\xi_i)
    \neq \vacuo \}.
  \end{gather*}
  We claim that this is an open subset of $X$.  This follows from the
  stratification theorem in \cite{heinzner-schwarz-stoetzel}. Indeed
  in the abelian case one can shift the gradient map so we can assume
  that $\xi_i =0 \in \lia$. Then $\Omega_i$ coincides with the stratum
  corresponding to the minimum of $||\mua||^2$, and as such it is
  open.  This proves the claim.  Finally observe that
  $W= \bigcap_{i=1}^k \Omega_i$.  Thus $W$ is also open in $X$ and (b)
  is proved.  Finally let $\sigma \subset P$ be a face of $P$. It is
  an exposed face, so there exists $\beta\in \lia$ such that
  \begin{gather*}
    \sigma=\{\xi\in P: \langle \xi,\beta \rangle = \max_{ \mua(X)}
    \langle \cdot,\beta \rangle \}.
  \end{gather*}
  Hence
  $\mua\meno(\sigma) = \{x \in X: \mup^\beta (x) = \max_X
  \mup^\beta\}$.
  By (b) there is $x\in W^\xi_{\mathrm{max}}$ such that
  $\mua(\overline{A \cdot x})=\mua(X)$. Define
  \begin{equation*}
    \phi_\infty: W^\xi_{\mathrm{max}} \lra \mu^{-1}(\sigma)
    \qquad\phi_\infty ( x ):= \lim_{t\mapsto +\infty} \exp(t\xi)\cdot x.
  \end{equation*}
  Since $\mup^{-1} (\sigma)$ is $A$-stable, it follows that
  $\overline{A \cdot \phi_{\infty} (x)}\subset \overline{A\cdot x}
  \cap \mu^{-1}(\sigma)$.
  On the other hand, let $a_n$ be a sequence of elements of $A$ such
  that
  $a_n \cd x \mapsto \theta \in \overline{A\cdot x} \cap \phi^{-1}
  (\sigma)$.  Since $\phi_\infty (\theta)=\theta$, it follows that
\begin{equation*}
  \theta=\lim_{n\mapsto \infty } \phi_\infty (a_n \cd x)
  =\lim_{n\mapsto \infty} a_n\cd \phi_\infty (x).
\end{equation*}
Therefore
\begin{equation*}
  \overline{A\cdot x} \cap \mu^{-1}(\sigma)=\overline{A \cdot \phi_{\infty} (x)}.
\end{equation*}
Since $\mua\restr{\overline{A\cd x}} : \overline{A\cd x} \ra P$ is a
surjective map,
$\mua ( \overline{A\cd x} \cap \mua\meno(\sigma) ) = \sigma$.  Thus
$\mua(\overline{A \cdot \phi_{\infty} (x)})=\sigma$.
\end{proof}

Let now $T^\C$ be a complex torus acting on the K\"ahler manifold
$Z$. The functions $\mu^\beta :Z \ra \R$ (for $\beta \in \liet$) are
Morse-Bott functions with even indices. Atiyah proved that the set of
their maximum points is a connected critical manifold.  Therefore the
corresponding unstable manifold, which coincides with the set
$\wm^\beta$ is an open dense subset of $Z$.  Set $\lia = i \liet $ and
$A=\exp(i\liet)$. Moreover $T^\C = A\cd T$, $Z^A = Z^T =Z^{T^\C}$
since the action is holomorphic. Finally $\mua = i \mu$ and
$\mu(T^\C \cd x) = -i \mua (A\cd x)$ since the $\mu$ is $T$-invariant.
Therefore the following theorem immediately follows from Proposition
\ref{caballo}.

\begin{teo}\label{main}
  Let $T$ be a compact torus.  Let $(Z,\omega)$ be a compact K\"ahler
  manifold on which $T^\C$ acts holomorphically. Assume that $ T$ acts
  in a Hamiltonian fashion with momentum map $\mu:Z \lra \liet^*$.
  Then there is a $T^\C$-orbit $\mathcal O$ such that
  $\mu(\overline{\mathcal O})=\mu(Z)$. More precisely:
  \begin{enumerate}
  \item the set $\{x\in Z:\, \mu(\overline{T^{\C} \cdot x})= \mu(Z)\}$
    is nonempty, open and dense;
  \item $\mu(Z)$ is a convex polytope with vertices in $\mu(Z^T)$;
  \item if $\,\sigma$ is a face of $\mu(Z )$, then there exists a
    $T^\C$-orbit $Y$ such that $\mu(\overline Y )=\sigma$.
  \end{enumerate}
\end{teo}
One can apply the method of proof used in Proposition \ref{caballo}
also in the setting considered by Heinzner and Huckleberry in
\cite{hhinv}. In this case $Z$ is a connected K\"ahler manifold, not
necessarily compact, and $X\subset Z$ is a compact irreducible
(complex) analytic subset.
\begin{teo}\label{main2}
  Let $X\subset Z$ be a compact irreducible (complex) analytic subset,
  which is invariant by the $T^\C$-action.  Then
  \begin{enumerate}
  \item $\mu(X)$ is a convex polytope with vertices in $\mu(X^T)$;
  \item the set
    $ W:=\{x\in X:\, \mu(\overline{T^{\C} \cdot x})= \mu(X)\}$ is
    nonempty, open and dense.
  \end{enumerate}
\end{teo}
\begin{proof}
  We claim that for $\xi\in \liet$ the set
  \begin{equation*}
    W_{\xi}:=\{x\in X: \overline{T^\C\cdot x } \cap \mu^{-1}(\xi) \neq\emptyset \}
  \end{equation*}
  is either empty or open and dense. Indeed by shifting we can assume
  that $\xi = 0$. Hence this is the set of semistable points for the
  action and the claim follows from the results in \cite{hhinv}.  Set
  $P:= \conv (\mu(X))$. This is a polytope with vertices in
  $\mu(X^T)$.  Let $\xi_1, \lds, \xi_k$ be the vertices.  Set
  $W':= W_{\xi_1}\cap \cds \cap W_{\xi_k}$.  This is an open dense
  subset of $X$.  Fix $x \in W'$.  By Theorem \ref{conv-orbite}
  $\mu(\overline{T^\C \cd x}) $ is a convex subset of $P$. Since it
  contains all the vertices we have
  $\mu(\overline{T^\C \cd x}) = \mu(X)=P$.  This proves (b) (which of
  course was proved directly also in \cite{hhinv}).  Moreover we have
  just seen that $W' \subset W$. The opposite inclusion is
  obvious. Hence $W=W'$ and (b) is proved.
\end{proof}

One would like to prove convexity for $\mua(X)$ for $X\subset Z$ a
general $A$-invariant closed submanifold of $Z$.  In this setting
convexity is unknown in general.  Convexity of $\mua(X)$ (and also
\emph{non-abelian} convexity) is known to hold if $X$ is a real flag
manifold, thanks to the pioneering paper \cite{kostant-convexity}, and
more generally if $Z$ is a Hodge manifold and $X$ is an irreducible
semi-algebraic subset of $Z$ whose real algebraic Zariski closure is
irreducible, \cite{bghc,heinzner-schuetzdeller}.

Using Proposition \ref{caballo} we can give a short argument when $X$
is a compact connected real analytic submanifold of $\PC$.  This class
is narrower than the one considered in \cite {heinzner-schuetzdeller},
but it is quite interesting. Above all, we feel that our proof is
rather geometric and very clear in its strategy.

\begin{lemma} \label{denso} Assume that $Z=\PC$ and that $X $ is a
  compact connected $A$-invariant real analytic submanifold endowed
  with the restriction of the Fubini-Study form.  Then for any
  $\beta \in \lia$ the intersection $\wm^\beta(X) \cap X$ is open and
  dense in $X$.
\end{lemma}
\begin{proof}
  Since $Z=\PC$, $\beta$ induces a linear flow on $\PC$ which
  restricts to the original one on $Z$ and $X$.  Assume that
  $v \in \su(n+1)$ is the infinitesimal generator of the linear flow
  and let $c_0 < \cds < c_r$ be the critical values of the function
  $f([z]):= i\sx v(z), z\xs / |z|^2$, that is the Hamiltonian of the
  flow on $\PC$.  Denote by $C_i$ the critical manifold corresponding
  to $c_i$ and let $W_i^u(\PC)$ be its unstable manifold. Then
  $\PC = \bigsqcup_{i=0}^r W_i^u$. Moreover for each $j $ the set
  $\bigsqcup_{i\leq j} W_i^u $ is equal to a linear subspace
  $L_j \subset \PC$. This is an elementary computation, see
  e.g. \cite[Lemma 7.4]{bgs}.  Since $\mua^\beta=f\restr{X}$ the
  critical points of $\mua^\beta$ on $X$ are given by
  $\bigcup_i (C_i\cap X)$.  If $\max_X \mua^\beta = c_j$, then
  $X \subset L_j$, $X$ is not contained in $L_{j-1}$ and
  $\wm^\beta(X) = W^u_j (\PC)\cap X = X - L_{j-1}$.  Assume by
  contradiction that this set is not dense in $X$. The $X\cap L_{j-1}$
  contains an open subset of $X$. Then $A:=(X\cap L_{j-1})^0$ is not
  empty. On the other hand $A \neq X$, since $X$ is not contained in
  $L_{j-1}$. Hence there is some point
  $x \in \partial A = \overline{A} - A$. Fix a real analytic chart
  $\phi: U \ra U' $ with $x\in U$ and $U'$ an open ball in
  $\R^k$. Locally around $x$ we have
  $L_{j-1} = \{ h_1 = \cds = h_p = 0\}$ for some local holomorphic
  functions $h_1, \lds, h_p$. Therefore the set
  \begin{gather*}
    U'':= \{ y\in U' : h_1\phi\meno (y) = \cds = h_p\phi\meno (y) =0
    \}
  \end{gather*}
  contains the open set $h ( A \cap U)$. Therefore $U''=U'$,
  $U \subset X \cap L_{j-1}$ and $x \in A$, a contradiction.
\end{proof}

Thanks to the previous lemma we can apply Proposition \ref{caballo}
and we get the following result.

\begin{teo}\label{main-real}
  Assume that $Z=\PC$ with the Fubini study metric.  Let
  $X \subset \PC$ be a compact connected $A$-invariant real analytic
  submanifold.  Then
  \begin{enumerate}
  \item $P=\mua(X)$ is a convex polytope with vertices in $\mu(X^A)$;
  \item the set $\{x\in X:\, \mua( \overline{A\cdot x})=\mua (X)\}$ is
    open and dense;
  \item if $\sigma \subset \mua (X)$ is a face of $P$ there exists a
    $A$-orbit $Y$ such that $\mua (\overline{Y} )=\sigma$.
  \end{enumerate}
\end{teo}

We remark that by \cite{bghc} the image of the gradient map is
independent of the K\"ahler metric within a fixed K\"ahler class.

\section{Action on the space of  measures}

Let $Z,X,G, K $ be as in \ref{gradient-cl}.  Denote by $\proba(X)$ the
set of Borel probability measures on $X$, which is a compact Hausdorff
space when endowed with the weak topology, see \cite{bgs,bz} for more
details and \cite{biliotti-ghigi-AIF,american,osaka,polar,israel} for background and
motivation.

Assume that $A=\exp(\lia)$ where $\lia \subset \liep$ is an Abelian
subalgebra.  Let $\Psi^A$ be the Kempf-Ness function for $ (X,A,\{e\}) $
as in Proposition \ref{Kempf-Ness-gradient-peter}.  Since $A$ acts on
$X$, we have an action on the probability measures on $X$ as follows:
\begin{gather*}
  A\times \pb \ra \pb , \quad (g, \nu) \mapsto g_*\nu.
\end{gather*}
In \cite{bz} it is proven that this action is continuous with respect
to the weak topology and that the function
\begin{gather}
  \label{defipsim} \PsiM : \proba(M) \times A \ra \R, \quad \PsiM(\nu,
  g) : = \int_M \Psi^A (x, g) d\nu(x),
\end{gather}
is a Kempf-Ness function for $(A,\proba(M),\{e\})$ in the sense of
Definition \ref{def-kn}.  Moreover, the gradient map is given by the
formula
\begin{gather}
  \fun : \proba(M) \ra \lia, \quad \fun (\nu) : = \int_M \mua(x)
  d\nu(x). \label{def-momento-misure}
\end{gather}
Since $\proba(X)$ is compact,
Theorem \ref{conv2} gives  a short  proof of the following result proved in
\cite{bilio-raffero}.
\begin{teo}
  Let $A=\exp(\lia)$ where $\lia \subset \liep$ is an Abelian
  subalgebra. If $\nu \in \proba(M)$, then
  \begin{enumerate}
  \item $\fun (A\cdot \nu)$ is a convex set.
  \item $\fun (\overline{A\cdot \nu})$ coincides with the convex hull
    of $\fun \big(\proba(M)^A \cap \overline{A\cdot \nu} \big)$, where
    $\proba(M)^A=\{\tilde \nu \in \proba (M):\, A\cdot \tilde
    \nu=\tilde \nu \}$.
  \end{enumerate}
\end{teo}

\end{document}